\documentclass{amsart}
\usepackage{amsmath,amssymb,amsthm,amsfonts}
\usepackage[english]{babel}
\usepackage[matrix,arrow]{xy}

\usepackage{dingbat}

\numberwithin{equation}{section}
\newcommand{\bC}{{\mathbb C}}

\newcommand{\bK}{{\mathbb K}}

\newcommand{\bN}{{\mathbb N}}

\newcommand{\bR}{{\mathbb R}}

\newcommand{\bZ}{{\mathbb Z}}

\newcommand{\tit}{\textit}
\newcommand{\sub}{\subset}
\newcommand{\wtilde}{\widetilde}
\newcommand{\bs}{\backslash}
\newcommand{\what}{\widehat}

\newtheorem{thm}{Theorem}[section]
\newtheorem{theorem}[thm]{Theorem}
\newtheorem{property}[thm]{Property}

\newtheorem{corollary}[thm]{Corollary}

\newtheorem{definition}[thm]{Definition}

\newtheorem{remark}[thm]{Remark}

\newtheorem{example}[thm]{Example}

\newtheorem{lemma}[thm]{Lemma}
\newtheorem*{claim}{Claim}

\begin{document}
\title[On the Theory of Stein Manifolds]{On the Theory of Stein Manifolds}
\author[Tran]{Dustin Tran}
\date{May 14, 2012}

\begin{abstract} This paper examines the broad structure on Stein manifolds and how it generalizes the notion of a domain of holomorphy in $\bC^n$. Along with this generalization, we see that Stein manifolds share key properties from domains of holomorphy, and we prove one of these major consequences. In particular, we investigate an equivalence, similar to domains of holomorphy and pseudoconvexity, on the class of manifolds. Then, we examine the canonical symplectic structure of Stein manifolds inherited from this equivalence, and how its symplectic topology develops. \end{abstract}
\maketitle

\section{Introduction} The class of Stein manifolds was first introduced by Karl Stein in 1951 \cite{St}, originally under the name of holomorphically complete manifolds. Based on a set of three axioms, Stein sought to generalize the concept of a domain of holomorphy in $\bC^n$ to complex manifolds. After initial developments, further contributions were made by H. Cartan \cite{C}, who proved vast generalizations of the first and second Cousin problems, now known as Cartan's Theorem A and B. Other notable results were discovered by K. Oka \cite{O2}, known previously for his solutions to the Cousin problems and his work on domains of holomorphy, and E. Bishop \cite{B}, who proved an embedding theorem for Stein manifolds. For the interested reader, R. Remmert's comprehensive survey article \cite{R} provides a thorough historical review.

This paper concerns Stein manifolds based on its primary motivation, and builds up some familiar properties from domains of holomorphy. In particular, we show that being Stein is equivalent to a certain notion of pseudoconvexity. Then by using the argument in the proof, we prove as a corollary that one of the axioms defining Stein manifolds is implicit from the other two. Lastly, we see how the admittance of a plurisubharmonic function leads us into questioning the natural symplectic structure that arises, laying down modern results on symplectic topology.

\section{Preliminary Concepts}
\subsection{Domains of Holomorphy, Holomorph-Convex, and Pseudoconvex.} We first review some preliminary notions prior to defining Stein manifolds, defining three equivalent definitions for a particular class of open sets in $\bC^n$. Given an open set $\Omega\sub\bC^n$, denote $A(\Omega)$ as the set of holomorphic functions on $\Omega$, $C^k(\Omega)$ as the class of $C^k$ functions on $\Omega$, and $C^k_{(p,q)}$ as the set of all forms of type $(p,q)$ with coefficients in $C^k.$

\begin{definition} An open set $\Omega\sub\bC^n$ is a domain of holomorphy if there do not exist non-empty open sets $U_1\sub\Omega$ and $U_2\sub\bC^n$, such that:
\begin{itemize}
\item $U_2$ is connected and $U_2\not\sub\Omega$;
\item For every $f\in A(\Omega)$, there is a function $g\in A(U_2)$ such that $f=g$ on $U_1$.
\end{itemize}\end{definition}
Informally, we can interpret a domain of holomorphy as an open set, where some $f\in A(\Omega)$ cannot be holomorphically extended to a bigger set.

For a compact subset $K\sub \Omega$, the \tit{holomorphically convex hull} of $K$ is defined as
\begin{equation}\label{eq:2.1} \what{K}:= \Big\{z\in \Omega:| f(z)|\leq \sup_{w \in K} \left| f(w) \right| \mbox{ for all } f \in A(\Omega)\Big\} .\end{equation}
\begin{definition} An open set $\Omega\sub\bC^n$ is \tit{holomorph-convex} if for every compact subset $K\sub\Omega$, its holomorphically convex hull $\what{K}\sub\Omega$ is compact.\end{definition}

Recall that a function $\phi:\Omega\to\bR$ is \tit{plurisubharmonic} if it is upper semicontinuous and $\tau\mapsto f(z+\tau w)$ is subharmonic for all $z,w\in\bC^n$. An equivalent definition on the class of $C^2$ functions is that $\phi\in C^2(\Omega)$ is plurisubharmonic if
\begin{equation}\label{eq:3.1}\sum_{j,k=1}^n \frac{\partial^2 \phi(z)}{\partial z_j\partial \overline{z}_k}w_j\overline{w}_k\geq 0,\hspace{5mm} z\in\Omega,w\in\bC^n.\end{equation} That is, the above hermitian form must be positive semi-definite. Moreover, we say that $\phi$ is \tit{strictly plurisubharmonic} if the hermitian form in \eqref{eq:3.1} is $>0$, i.e., positive definite.

\begin{definition} An open set $\Omega\sub\bC^n$ is pseudoconvex if there exists a plurisubharmonic function $\phi\in C^1(\Omega)$ such that
\begin{itemize}
\item For all $c\in\bR$, $\Omega_c:=\{z\in\Omega:\phi(z)<c\}$ is relatively compact in $\Omega$.\end{itemize}\end{definition}
Moreover, we say that $\Omega$ is \tit{strictly pseudoconvex} if the function $\phi$ is strictly plurisubharmonic. We now introduce the following theorem which relates these three classes together.

\begin{theorem}\label{th:C-T} Given an open set $\Omega\sub\bC^n$, the following conditions are equivalent:
\begin{itemize}
\item $\Omega$ is a domain of holomorphy.
\item $\Omega$ is holomorph-convex.
\item $\Omega$ is pseudoconvex. \end{itemize}\end{theorem}
Readers interested in the proof may consult \cite[Theorem 7.3.2]{Sch}, which uses standard metric techniques. Although we do not use Theorem \ref{th:C-T} to prove the results in this paper, the theorem provides a powerful intuition in recognizing the implicit structure from Stein manifolds.

\begin{remark} This equivalence is actually part of a larger set of equivalences, including Levi convexity and what is now known as the local Levi property. When the equivalence of these five conditions was first proven, the main difficulty lied in proving that the local Levi property implies domain of holomorphy. This became known as the Levi problem, named after E.E. Levi in 1911, and it remained unsolved until 1953, first proven by K. Oka \cite{O}. \end{remark}

\subsection{Stein manifolds.} We now provide the original definition of Stein manifolds, and later investigate what equivalent definitions come about. Recall that a \tit{complex manifold} $X^n$ is a real $2n$-manifold (so $n=\dim_{\bC}X$), equipped with holomorphic transition maps.

\begin{definition}A Stein manifold is a complex manifold $X^n$ which satisfies the following axioms:
\begin{itemize}
\item Convexity: $X$ is holomorph-convex;
\item Separation: If $z_1\neq z_2\in X$, then there exists $f\in A(X)$ such that $f(z_1)\neq f(z_2);$
\item Local Coordinates: For all $z\in X$, there exists $n$ functions $f_1,...,f_n\in A(X)$ which form a coordinate system at $z$.
\end{itemize}\end{definition}

By the definition, it is easy to see that every closed complex submanifold of a Stein manifold is Stein, and that the Cartesian product $X_1\times X_2$ of two Stein manifolds is also Stein.

\begin{example} Let $X^1$ be a noncompact Riemann surface (complex manifold of dimension one). Then by Behnke-Stein $(1948)$, $X^1$ must be a Stein manifold, having proved the result by using a version of Runge's approximation theorem.\end{example}

\begin{example} Every domain of holomorphy in $\bC^n$ is a Stein manifold $X^n$. \end{example}
Indeed, because the structure on domains of holomorphy is an underlying motivation for Stein manifolds, many theorems concerning domains of holomorphy apply equally well to Stein manifolds. First, we recall the following theorem, which is just the generalization of the same argument for open sets in $\bC.$
\begin{theorem} \label{th:3.1}Let $\Omega$ be an open set in $\bC^n$. For every compact set $K\sub\Omega$ and every neighborhood $U$ of $K$, there exist constants $C_{\alpha}$ for all multi-orders $\alpha$ such that
\begin{equation}   \sup_K|\partial^\alpha u|\leq C_\alpha||u||_{L^1(U)},\hspace{5mm}u\in A(\Omega).\end{equation}\end{theorem}

We now note the following properties, which naturally extend to Stein manifolds by analogous arguments from open sets in $\bC^n$. For brevity, we omit the exact proofs, though the analogous arguments are more or less straightforward.
\begin{theorem} \label{th:3.2}Let $\Omega$ be an open set in $\bC^n$, or $\Omega^n$ a complex manifold. Also, suppose $\phi$ is a strictly plurisubharmonic function in $C^\infty(\Omega)$ such that $\Omega_c:=\{z\in\Omega:\phi(z)\leq c\}$ is relatively compact in $\Omega$ for all $c\in\bR.$ Then every function which is holomorphic in a neighborhood of $\Omega_0$ can be approximated uniformly over $\Omega_0$ by functions in $A(\Omega)$.\end{theorem}

\begin{theorem} \label{th:3.3} Let $\Omega$ is pseudoconvex open set in $\bC^n$, or $\Omega^n$ a complex manifold such that there exists a strictly plurisubharmonic function $\phi\in C^\infty(X)$ such that $\Omega_c:=\{z\in X:\phi(z)<c\}\sub\sub X$ for every $c\in\bR.$ Then the equation $\overline{\partial}u=f$ has a solution $u\in C^{\infty}_{(p,q)}(X)$ for every $f\in C^\infty_{(p,q+1)}(X)$ such that $\overline{\partial}f=0.$ \end{theorem}

Readers interested in their proofs may look at \cite[Theorems 4.3.1 and 5.2.8]{H} and \cite[Corollaries 4.2.6 and 5.2.6]{H} respectively.

\section{Stein Manifolds and Pseudoconvexity}
We now prove a non-trivial extension of a property on domains of holomorphy. In particular, we shall establish a equivalence of definitions, similar to domains of holomorphy and pseudoconvexity, in the sense of manifolds. Not surprisingly, this equivalence requires smoothness and nondegeneracy on the plurisubharmonic function, as opposed to only continuity for $\bC^n$.

\begin{lemma} \label{le:3.1}If $X^n$ is a Stein manifold, then $X$ admits a strictly plurisubharmonic function $\phi\in C^{\infty}(X)$ such that
\begin{itemize}
\item For all $c\in\bR$, $X_c:=\{z\in X:\phi(z)<c\}$ is relatively compact in $X$.
\end{itemize}\end{lemma}
\begin{proof}
Suppose $X^n$ is a Stein manifold, and let $K\sub X$ be compact and $U$ an open neighborhood of $\what{K}$, the holomorphically convex hull of $K$. By the convexity axiom on $X$, we can choose a sequence of compact subsets of $\Omega$,
\begin{equation}K_1=\what{K}\sub K_2=\what{K_2}\sub K_3=\what{K_3}\sub...,\end{equation}
such that $\what{K_j}=K_j$ and $\cup K_j=X$ for all $j\in\bN.$ Let $U_j$ be open sets such that $K_j\sub U_j\sub K_{j+1}$ and $U_1\sub U.$

Since $\what K_j=K_j$, then for all $j$, we can choose functions $f^j_k\in A(X)$ for $k=1,...,k_j$, such that
\begin{equation}\label{eq:3.2}|f^j_k|<1\implies \sum_{k=1}^{k_j} |f^j_k(z)|^2<2^{-j},\hspace{5mm}z\in K_j\end{equation}\vspace{-2mm}
\begin{equation}\label{eq:3.3}\max_k|f^j_k(z)|>1\implies \sum_{k=1}^{k_j} |f^j_k(z)|^2>j,\hspace{5mm}z\in K_{j+2}\bs U_j.\end{equation}
We note the implications are true, by taking sufficiently large powers of $f^j_k$ and rearranging. Now by the local coordinate axiom, we can also choose $n$ functions forming a system of local coordinates at any $z\in K_j$. Consider the function
\begin{equation}\phi:X\to\bR,\end{equation}
$$z\longmapsto\Big(\sum_{j\in\bN}\sum_{k=1}^{k_j}|f^j_k(z)|^2\Big)-1.$$
By \eqref{eq:3.2}, the inner series is bounded by $2^{-j}$, so the full series converges, and by \eqref{eq:3.3}, we see that $\phi(z)>j-1$ for all $z\in X\bs U_j$. Note that the series
\begin{equation}\sum_{j,k}f^j_k(z)\overline{f^j_k(\zeta)}\end{equation}
converges uniformly on compact subsets of $X\times X$, so the sum is holomorphic on $z$ and its complex conjugate is holomorphic on $\zeta$. Then $\phi$ is a smooth function. By above, it is also easy to see that $\phi$ is plurisubharmonic.

It remains to show that $\phi$ is strictly plurisubharmonic. Suppose for some $z_0\in X$,
\begin{equation}\label{eq:3.4}\sum_{i=1}^n w_i\frac{\partial f^j_k(z)}{\partial z_{i}}=0\hspace{5mm}\text{ for all $j,k$.}\end{equation}
Then because there always exists $n$ functions $f^j_k$ forming a local coordinate system at $z$, $w=0$. Thus, \eqref{eq:3.4} is positive-definite, so we are done.
\end{proof}
\begin{remark}\label{re:3.1} Note that the proof does not require the separation axiom, only using the holomorph-convex and local coordinates axioms. This will become important when establishing an equivalence between the two conditions of the lemma, as it will prove that the separation axiom is not necessary in defining Stein manifolds.\end{remark}

We now prove the implication from the other side, in order to obtain the following theorem. Note that the proof essentially mimicks the argument seen in \cite{H}.

\begin{theorem} \label{th:3.5} A complex manifold $X^n$ is a Stein manifold if and only if it admits a strictly plurisubharmonic function $\phi\in C^{\infty}(X)$ such that
\begin{itemize}
\item For all $c\in\bR$, $X_c:=\{z\in X:\phi(z)<c\}$ is relatively compact in $X$.\end{itemize}\end{theorem}
\begin{proof}
By Lemma \ref{le:3.1}, we have established the forward implication. Now suppose we have such a strictly plurisubharmonic function $\phi\in C^\infty(X)$. We first assert the following claim.

\begin{claim} For all $w_0\in X$, there exists a neighborhood $U$ and function $f\in A(U)$, such that $f(w_0)=0$ and
\begin{equation} \text{Re}f(z)<\phi(z)-\phi(w_0),\hspace{5mm} z\neq w_0\in U.\end{equation}\end{claim}

\begin{proof} We examine $w_0$ in local coordinates and obtain the inequality by looking at its Taylor expansion. Let $(z_1,..,z_n)$ represent the local coordinates at $w_0$ in a neighborhood $U$, such that the corresponding coordinate for $w_0$ is the origin. Then by Taylor series expansion, we can represent $\phi$ as
\begin{equation}\label{eq:3.9} \phi(z)=\phi(0)+ \text{Re}f(z)+\sum_{j,k=1}^n \frac{\partial^2\phi(0)}{\partial z_j\partial\overline{z}_j}z_j\overline{z}_k + O(|z|^3),\end{equation}
where $f$ is a polynomial of degree $\leq 2$ and $f(0)=0$, represented by $f(z)=az^2+bz$. By definition of $\phi$, the hermitian form in \eqref{eq:3.9} is positive definite, so for non-zero $z\in U$,
\begin{equation} \phi(z)>\phi(0)+\text{Re}f(z)\implies \text{Re}f(z)<\phi(z)-\phi(w_0).\end{equation}\end{proof}

Now let $w_0\in X$. By the above claim, there exists a neighborhood $U$ and function $g_0\in A(U),$ so that $w_1\not\in U$ and $w_0$ is covered by a set of local coordinates. Also, let $U_1$, $U_2$ be other neighborhoods of $w_0$, such that $U_1$ is relatively compact in $U_2$, and $U_2$ is relatively compact in $U.$ Let $\psi$ be a smooth function compactly supported on $X$ by $U_2$ and $\psi(U_1)=1.$

We aim to use $\psi$ as a smooth cutoff function. Note that supp$(\overline{\partial}\psi)\in\overline{U}_2\bs U_1,$ so there exists $c>\psi(w_0)$ and $\epsilon>0$, such that
\begin{equation} \label{eq:3.10}\text{Re }g_0(z)<-\epsilon,\hspace{5mm} z\in\text{supp}(\overline{\partial}\psi)\text{ and }\phi(z)<c.\end{equation}

From the proof of Theorem \ref{th:3.3} \cite[p.126]{H}, there exists a function $\phi_c\in C^{\infty}(X_c),$ bounded from below in $X_c$, such that the solution $\overline{\partial}u=f$ for every $f\in L^2_{0,1}(X_c,\phi_c)$ with $\overline{\partial}f=0$ has a solution $u\in L^2(X_c,\phi_c)$ where
\begin{equation}\label{eq:3.11}||u||_{\phi_c}\leq||f||_{\phi_c}.\end{equation}
Also by Theorem \ref{th:3.3}, note that $u$ must be smooth if $f$ is. Now let $g\in A(U)$ and consider with a sufficiently large parameter $t$,
\begin{equation} g_t=\psi ge^{tg_0}-u_t,\end{equation}
where we aim to construct $u_t$ such that $g_t\in A(X_c)$ and $u_t$ is small. Note that $g_t$ is holomorphic if
\begin{equation}\label{eq:3.12} \overline\partial(\psi ge^{tg_0}-u_t)=0\implies\overline\partial u_t=ge^{tg_0}\overline\partial\psi.\end{equation}
Let $f_t:=\overline\partial u_t.$ Then by \eqref{eq:3.10}, $||f_t||_{\psi_c}=O(e^{-\epsilon t})$ for fixed $g$, so we can apply \eqref{eq:3.11}. Thus, $f_t$ has a solution $u_t$ such that
\begin{equation}\label{eq:3.13} ||g_t||_{\psi_c}=O(e^{-\epsilon t}),\hspace{5mm}t\to\infty.\end{equation}
Note that by \eqref{eq:3.12},  $u_t$ is holomorphic in the complement of the support of $\overline\partial \psi$ in $X_c$. By Theorem \ref{th:3.1}, we obtain for $w_1\neq w_0\in X$,
 \begin{equation}g_t(w_1)=u_t(w_1)\to0\text{ and }g_t(w_0)=g(w_0)+u_t(w_0)\to g(w_0),\hspace{5mm} t\to\infty.\end{equation}
Without loss of generality, let $g(w_0)=1$, so $g_t(w_1)\neq g_t(w_0)$ for $t$ sufficently large. By Theorem \ref{th:3.2}, we can locally approximate $g_t$, on $\overline{X_{\phi(w_0)}}$, by a function $h\in A(X)$ such that $h(w_0)\neq h(w_1)$. Thus, we have satisfied the separation axiom.

Moreover, \eqref{eq:3.13} implies that for every $c<\phi(w_0)\in\bR,$
\begin{equation}\label{eq:3.14}\int_{X_c}|g_t|^2dV\to0,\hspace{5mm}t\to\infty.\end{equation}
By Theorem \ref{th:3.1}, $g_t\to0$ uniformly on compact subsets of $X_c$. If $c'<c$, then $|g_t|<\frac{1}{2}$ on $X_c$ for sufficiently large $t$ while $g_t(w_0)\to 1$. By Theorem \ref{th:3.2}, we can approximate $g_t$ again by functions in $A(X)$, so $w_0$ cannot belong in $\what{\overline{\Omega}}_{c'}$, the holomorphically convex hull of $\overline{\Omega}_{c'}$, for any $c'<\phi(w_0).$ Thus, $\what{\overline{\Omega}}_{c'}=\overline{\Omega}_{c'}$ for all $c'$. Thus, we have satisfied the holomorph-convexity axiom.

We now prove local coordinates axiom. Without loss of generality, we can suppose $g\in A(U)$ vanishes at $w_0$. Then because $\partial g_t=\partial g-\partial u_t$ at $w_0$ and $\partial u_t\to0$ at $w_0$ by Theorem \ref{th:3.1}, we have that $\partial g_t\to\partial g$ at $w_0$ when $t\to\infty.$

Now let $g^1,...,g^n$ define a coordinate system at $w_0$ formed by functions which vanish there. Then the Jacobian of the corresponding functions $g_t^1,...,g_t^n\in A(X_c)$ with respect to $g^1,...,g^n$ converges to $1$ as $t\to\infty.$ Now by Theorem \ref{th:3.2}, we can approximate $g_t^1,...,g_t^n$ by $h^1,...,h^n\in A(X)$ such that the Jacobian of $h^1,...,h^n$ with respect to $g^1,...,g^n$ is also non-zero at $w_0$. This proves the last axiom.
\end{proof}

Thus, a generalized version of Theorem \ref{th:C-T} still holds on the class of manifolds. This becomes important in deciphering how other problems belonging to several complex variables in $\bC^n$ may translate more loosely to manifolds and sheaves, becoming instrumental in Cartan's later work in particular. We also see that with the admittance of a particularly nice function, we obtain a canonical symplectic structure, built from the manifold's complex structure and plurisubharmonic function. Such consequences shall be delved into deeper within the following section. But first, we note a remarkable property on Stein manifolds.

\begin{corollary} The separation axiom is a consequence of the other two axioms, and so Stein manifolds are classified by holomorph-convexity and how it evaluates points locally.\end{corollary}
\begin{proof} Let $X$ be a complex manifold satisfying the holomorph-convex and local coordinates axiom. As explained in Remark \ref{re:3.1}, we can still apply Lemma \ref{le:3.1} since its proof does not use the separation axiom. Then by Theorem \ref{th:3.5}, $X$ must be Stein. \end{proof}

\section{A Glimpse of its Symplectic Structure}

\subsection{Liouville domains and manifolds.}
 By the admittance of a strictly plurisubharmonic function, Stein manifolds fall into a broader class of manifolds, one which has seen much development in the past decade. Here, we denote $M^n$ to be a manifold of real dimension $n$. Recall that if a symplectic manifold $M^n$ is exact, then its symplectic form $\omega_M$ has a corresponding \tit{primitive} $\theta_M$. That is, there exists a $1$-form $\theta_M$ satisfying $\omega_M=d\theta_M$. The \tit{Liouville vector field} $Z_M$ is the vector field such that $\omega_M(Z_M,\cdot)=\theta_M$.
\begin{definition}A Liouville domain is a compact exact symplectic manifold with boundary $M^{2n}$, such that $Z_M$ points strictly outwards along $\partial M$. \end{definition}

Given a manifold $M^{2n}$, $h:M\to\bR$ is an \tit{exhausting function} if it is bounded below and proper (continuous such that the inverse image of compact subsets is compact). An exact symplectic structure is \tit{complete} if the flow of $Z_M$ exists for all time.
\begin{definition}A Liouville manifold is a complete exact symplectic manifold $M^{2n}$ which admits an exhausting function $h:M\to\bR$ with the following property:

There exists a sequence $\{c_k\}_{k\in\bN}$ for $c_k\in\bR$, approaching $+\infty$ as $k\to\infty$, such that $dh(Z_M)>0$ along $h^{-1}(c_k)$.\end{definition}
We say that a Liouville manifold $M^{2n}$ has \tit{finite type} if $dh(Z_M)>0$ outside a compact subset $C\sub M$. Note that for any Liouville manifold $M^{2n}$, its corresponding sublevel sets yield an exhaustion of $M$ by Liouville domains.

\begin{property} \label{pr:4.1}A Stein manifold $X^n$ is a Liouville manifold $\wtilde{M}^{2n}$. Moreover, if its plurisubharmonic function can be taken to have a compact critical set, then the Stein manifold is of finite type.\end{property}
\begin{proof}
Now let $X^n$ be a Stein manifold. By Theorem \ref{th:3.5}, $X$ admits a smooth strictly plurisubharmonic function $\phi: X \to \bR$, such that each open set $X_c$ is relatively compact. Note that this implies $\phi$ is our desired exhausting function. Furthermore, $X$ has a $1$-form
\begin{equation}\theta_X:=-d^c\phi,\end{equation}
 where $d^c$ is defined by $d^c(\phi)(X):=d\phi(JX)$, $J$ being the complex structure which evaluates cotangent vectors. $X$ then carries the canonical symplectic form with primitive $\theta_X$, so
 \begin{equation}\omega_X:=-dd^c\phi.\end{equation}
 Fixing $c_k$ to be a regular value of $\phi$, the sublevel set
\begin{equation} \label{eq:sublevel}
M^{2n} := \phi^{-1}(-\infty,c_k],
\end{equation}
equipped with $\omega_X$, is a Liouville domain. Thus, every Stein manifold $X^n$ is a Liouville manifold $\wtilde{M}^{2n}$. The last remark regarding the critical set of $\phi$ follows by construction.\end{proof}

\subsection{Symplectic cohomology.}
We now introduce a particularly useful invariant for distinguishing symplectic manifolds in one of the Liouville classes, first considering Liouville domains. We denote $SH^*$ as short-hand for symplectic cohomology. For a more comprehensive introduction to symplectic cohomology, one may consult Seidel's survey article \cite{S}. Fix $\bK$ to be the desired field of coefficients.
\begin{definition}If $M$ is a Liouville domain of dimension $2n$, then the symplectic cohomology of $M$ with $\bK$-coefficients, $SH^*(M)$, is a $\bZ/2$-graded $\bK$-vector space with a natural $\bZ/2$-graded map
\begin{equation}H^{*+n}(M;\bK)\to SH^*(M).\end{equation}\end{definition}

Given two Liouville domains $U$, $M$ where $\dim U=\dim M$, and an embedding $\epsilon:U\to M$ such that $\epsilon^*\theta_M-c\theta_U$ is exact for some constant $c>0$, then Viterbo's construction \cite{V} assigns to each embedding a pull-back restriction map
\begin{equation}SH^*(\epsilon): SH^*(M)\to SH^*(U).\end{equation}
This is homotopy invariant within the space of all such embeddings, and functorial with respect to composition of embeddings. More generally, by a parametrization argument of Viterbo's construction, this is invariant under isotopies of embeddings, within the same class.

Note that any Liouville domain can be canonically extended to a finite type Liouville manifold by attaching an infinite cone to the boundary. Conversely, any finite type Liouville manifold can be truncated to a Liouville domain which is a sufficiently large sublevel set of the exhausting function. So with the natural correspondence between Liouville domains and Liouville manifolds of finite type, we can naturally define symplectic cohomology on the latter. We aim to generalize this feature for Liouville manifolds of finite type to arbitrary Liouville manifolds.
\begin{definition}If $M$ is a Liouville manifold, then
\begin{equation}\label{eq2} SH^*(M)=\lim_{\leftarrow}SH^*(U),\end{equation}
where the limit is over all pairs $(U,k)$ such that the Liouville vector field $Z_M$ satisfying $\theta_M+dk$ points strictly outwards along $\partial U$, where $U\sub M$ is a compact codimension zero submanifold and $k$ is a function on $M$.  \end{definition}

\subsection{Exotic symplectic structures.} Note that any smooth complex affine algebraic variety is a Stein manifold. By choosing a suitable K\"{a}hler form, we see that smooth complex affine algebraic varieties naturally correspond to Liouville manifolds of finite type. The most recent result on exotic symplectic structures, by Abouzaid-Seidel \cite{A-S} in 2010, shows that there exists infinitely many distinct exotic structures on smooth complex affine varieties of real dimension $n\geq 6$.

However, as seen in Property \ref{pr:4.1}, while Stein manifolds correspond to Liouville manifolds, the outcome may not always be of finite type. Thus, this result cannot be said of Stein manifolds in general. Regardless, for the class of Stein manifolds, we still have the following result by Mc-Lean \cite{M-L} from 2009, which expands on Seidel-Smith's result \cite{S-S} from 2005.

\begin{theorem}[Mc-Lean] There exists an exotic Stein manifold for all complex dimension $n> 2$. Moreover, for complex dimension $n> 3$, there exists infinitely many exotic Stein manifolds of finite type, each of which are pairwise distinct. \end{theorem}

The main tool used in distinguishing exotic structures from the standard one has been using symplectic (co)homology as an invariant on particular classes of manifolds. Abouzaid-Seidel's argument for their paper uses the fact that symplectic cohomology is invariant under symplectomorphisms between Liouville manifolds of finite type, yet it remains open whether or not this is true for Liouville manifolds in general. The proof of the invariance for finite type relies heavily on the cylindrical end property exclusive to finite type Liouville manifolds, and so no possible extension of the current proof can be taken.

Other symplectic invariants remain useful as well. Harris \cite{Ha} proves the existence of exotic structures by examining how some symplectic properties change under small deformations of its symplectic form. He also shows that symplectic cohomology fails to distinguish them from the standard form, while this invariant succeeds. Ultimately, while symplectic cohomology continues to be a powerful tool, modern developments in other tools and concepts are in the works, in order to discover new properties from a new perspective.

\subsection{Last remarks.}
Stein manifolds provide a rich algebraic, complex, and symplectic structure, all of which stem from the manifold's primary motivation as a domain of holomorphy. For open sets in $\bC$, we see that there always exists a holomorphic function on the open set, such that it cannnot be holomorphically extended onto a bigger set. This is not true for $\bC^n$ in general, and so domains of holomorphy become our natural consideration for this nice structure of open sets in $\bC$.

In turn, Stein manifolds become the natural consideration for this structure on the class of manifolds. From this generalization, we obtain many of the same properties, thus connecting many of the general concepts regarding sheaves, varieties, and manifolds, back down to the more concrete, identifiable spaces with already heavily-developed theorems and properties. As a result, Stein manifolds remain an incredibly useful space to analyze, bridging the gap between two seemingly separate categories.

\end{document}